\documentclass[12pt,letterpaper,titlepage]{amsart}
\usepackage{amsmath, amssymb, amsthm, amsfonts,amscd,amsaddr,wasysym}
\usepackage[
    paper=a4paper,
    portrait=true,
    textwidth=425pt,
    textheight=650pt,
         tmargin=3cm,
    marginratio=1:1
            ]{geometry}

\usepackage[cmtip,matrix,arrow]{xy} \SelectTips{cm}{10}

\newtheorem{thm}{Theorem}[section]
\newtheorem{cor}[thm]{Corollary}
\newtheorem{lem}[thm]{Lemma}
\newtheorem{pro}[thm]{Proposition}

\numberwithin{equation}{section}

\theoremstyle{definition}

\theoremstyle{remark}
\newtheorem{rem}[thm]{Remark}

\DeclareMathOperator{\tr}{tr} \DeclareMathOperator{\ad}{ad}

\newcommand{\e}{\mathbf{1}}

\newcommand{\C}{\mathbb{C}}

\newcommand{\D}{\Delta}
\newcommand{\SD}{{\sqrt{\Delta}}}
\newcommand{\eps}{\varepsilon}

\newcommand{\func}{\varphi}
\newcommand{\g}{\mathfrak{g}}

\newcommand{\metric}{\mathbf{g}}

\newcommand{\N}{\mathbb{N}}

\newcommand{\R}{\mathbb{R}}

\newcommand{\RnG}{{\mathcal{R}_R(G)}}
\newcommand{\U}{\mathcal{U}}
\newcommand{\VG}{{\mathcal{V}(G)}}

\newcommand{\WRpT}{{\mathcal{W}_{R',\vartheta}}}

\newcommand{\GL}{\operatorname{GL}}
\newcommand{\Hc}{\mathcal{H}}
\newcommand{\gf}{\mathfrak{g}}

\usepackage[usenames]{color}

\title{On Sobolev norms for Lie group representations}

\subjclass[2000]{}
\begin{document}
\date{ \today }

\begin{abstract} We define Sobolev norms \textcolor{black}{of arbitrary real order} for a Banach representation
$(\pi, E)$ of a Lie group, with regard to a single differential operator $D=d\pi(R^2+\Delta)$.
Here, $\Delta$ is a Laplace element in the universal enveloping algebra, and $R>0$ depends
explicitly on the growth rate of the representation. In particular, we obtain a spectral gap for $D$ on the space of smooth vectors of $E$. The main tool is a novel factorization of the delta distribution on a Lie group.
\end{abstract}

\author[Gimperlein]{Heiko Gimperlein}
\email{h.gimperlein@hw.ac.uk}
\address{Maxwell Institute for Mathematical Sciences and Department of Mathematics, Heriot--Watt University, Edinburgh, EH14 4AS, United Kingdom}

\author[Kr\"otz]{Bernhard Kr\"{o}tz}
\email{bkroetz@gmx.de}
\address{Institut f\"ur Mathematik, Universit\"at Paderborn,\\ Warburger Stra\ss e 100,
33098 Paderborn}

\maketitle

\section{Introduction}

Let $G$ be a Lie group and $(\pi,E)$ be a Banach representation of $G$, that is,
a morphism of groups $\pi: G \to \GL(E)$ such that the orbit maps
$$ \gamma_v: G \to E, \ \ g\mapsto \pi(g)v ,$$
are continuous for all $v\in E$.
\par We say that a vector $v$ is $k$-times differentiable if $\gamma_v \in C^k(G, E)$
and write $E^k\subset E$ for the corresponding subspace. The smooth vectors are then defined by $E^\infty=\bigcap_{k=0}^\infty E^k$.
\par The space $E^k$ carries a natural Banach structure: if $p$ is a defining norm
for the Banach structure on $E$, then a $k$-th Sobolev norm of $p$ on $E^k$
is defined as follows:

\begin{equation} \label{def StSob} p_k (v):=\Big(\sum_{m_1 +\ldots +m_n\leq k} p(d\pi(X_1^{m_1}\cdot \ldots \cdot X_n^{m_n})v)^2\Big)^{1\over 2}\quad (v \in E^k)\, .\end{equation}
Here $X_1, \ldots, X_n$ is a fixed basis for the Lie algebra $\gf$ of $G$, and
$d\pi: \U(\gf)\to \operatorname{End}(E^\infty)$  is, as usual, the derived representation for the
universal enveloping algebra $\U(\gf)$ of $\gf$. Then $E^k$, endowed with the norm $p_k$, is
a Banach space and defines a Banach representation of $G$.  Furthermore, $E^\infty$ carries a natural Fr\'echet structure, induced by the Sobolev norms $(p_k)_{k \in \N_0}$.
The corresponding $G$-action on $E^\infty$ is smooth and of moderate growth, i.e.~an
$SF$-representation in the terminology of  \cite{bk}.

\par In case $(\pi, \Hc)$ is a unitary representation on a Hilbert space $\Hc$, there is
an efficient  way to define the Fr\'echet structure on $\Hc^\infty$
via a Laplace element

\begin{equation}\label{def D1} \Delta =  -\sum_{j=1}^n X_j^2\end{equation}
in $\U(\gf)$.  More specifically, one defines the $2k$-th Laplace Sobolev norm in this case
by
\begin{equation}\label{Lap-Sob} {}^\Delta p_{2k} (v) := p (d\pi(({\bf 1}+ \Delta )^k) v) \qquad (v\in E^{2k})\, .\end{equation}
The unitarity of the action then implies that the standard Sobolev norm
$p_{2k}$ is equivalent to ${}^\Delta p_{2k}$.

\par For a general Banach representation $(\pi, E)$ we still have
$E^\infty=\bigcap_{k=0}^\infty \operatorname{dom}(d\pi(\Delta^k))$, but it is no longer true that
${}^\Delta p_{2k}$, as defined in \eqref{Lap-Sob}, is equivalent to $p_{2k}$:
it typically  fails that $p_{2k}$ is dominated by $^{\Delta}p_{2k}$, for example if
$-1\in \operatorname{spec} (d\pi(\Delta))$ or if elliptic regularity fails as in Remark \ref{regremark} below.

\par In the following we use $\Delta$ for the expression \eqref{def Delta}, a first-order
modification of $\Delta$ as defined in \eqref{def D1}, in order  to make $\Delta$ selfadjoint
on $L^2(G)$. In case $G$ is unimodular, we remark that the two notions \eqref{def Delta} and
\eqref{def D1}  coincide.

One of the main results of this note is that every Banach representation
$(\pi, E)$ admits a constant $R=R(E)>0$ such that the operator $d\pi(R^2+ \Delta) : E^\infty \to E^\infty$ is invertible, see Corollary \ref{cor spectral gap}. The constant $R$ is closely
related to the growth rate of the representation, i.e.~the growth of the weight
$w_\pi(g)=\|\pi(g)\|$.

More precisely, for the Laplace Sobolev norms defined as
\begin{equation}\label{Rost1} {}^\Delta p_{2k} (v) := p (d\pi((R^2 + \Delta)^k) v) \qquad (v\in E^{2k})\, ,\end{equation}
we show that the families $(p_{2k})_k $ and $\left({}^\Delta p_{2k}\right)_k$ are equivalent in the
following sense: Let $m$ be the smallest even integer greater \textcolor{black}{or} equal to $1+\dim G$. Then there
exist constants $c_k, C_k>0$ such that
$$ c_k \cdot {}^\Delta p_{2k}(v)\leq p_{2k}(v) \leq C_k \cdot {}^\Delta p_{2k+m}(v)\qquad (v\in E^\infty)\, .$$

\par The above mentioned results follow from a novel factorization
of the delta dis\-tri\-bution $\delta_{\bf 1}$  on $G$, see Proposition \ref{factorize} in the main text
for the more technical statement.
This in turn is a consequence of the functional calculus for $\sqrt{\Delta}$, developed
in \cite{cgt}, and previously applied to representation theory  in \cite{gkl} to derive factorization results for
analytic vectors.
The functional calculus allows us to define Laplace Sobolev norms for any order $s\in \R$ by
$${}^\Delta p_{s} (v) := p (d\pi((R^2+\Delta)^{s\over 2}) v) \qquad (v \in E^\infty)\,.$$
On the other hand \cite{bk} provided another definition of Sobolev norms
 for any order $s\in \R$; they were denoted $Sp_s$ and termed induced Sobolev norms
there. The norms $Sp_s$ were based on a noncanonical localization to a neighborhood of $\e \in G$, identified with the unit ball in $\R^n$, and used the $s$-Sobolev norm on $\R^n$. We show that the two notions ${}^\Delta p_{s}$ and $Sp_s$ are equivalent up to
constant shift in the parameter $s$, see Proposition \ref{prop compare}. The more geometrically defined norms ${}^\Delta p_{s}$ may therefore replace the norms $Sp_s$ in \cite{bk}.

 \par Our motivation for this note stems from harmonic analysis on homogeneous spaces, see
 for example \cite{b} and \cite{dkks}.
 Here one  encounters naturally the dual representation of some $E^k$ and in this context it is often quite cumbersome to estimate the dual norm of
$p_k$, caused by the many terms in the definition \eqref{def StSob}. On the other hand
the dual norm of ${}^\Delta p_s$, as defined by one operator $d\pi((R^2+\Delta)^{s\over 2})$, is easy to control and simplifies a variety of technical issues.

\section{Some geometric analysis on Lie groups}\label{SectionGeometricAnalysis}

Let $G$ be a Lie group of dimension $n$ and $\metric$ a left invariant Riemannian metric on $G$.
The Riemannian measure $dg$ is a left
invariant Haar measure on $G$.
We denote by $d(g,h)$ the distance function associated to $\metric$
(i.e.~the infimum of the lengths of all paths connecting group elements $g$ and $h$), by $B_r(g) = \{x \in G\ |\  d(x,g)<r\}$ the ball of radius $r$ centred at $g$,
and we set
$$ d(g) :=d(g,\e) \qquad (g\in G) \, .$$
Here are two key properties of $d(g)$, which will be relevant later, see \cite{garding}:
\begin{lem}\label{lem=sm} If
$w : G \to \R_+$ is locally bounded and submultiplicative (i.e.~$w(gh) \leq w(g)w(h)$), then there exist
$c_1,C_1>0$ such that
$$w(g) \leq C_1 e^{c_1 d(g)} \qquad (g\in G)\ .$$
\end{lem}

\begin{lem} There exists $c_G>0$ such that for all $C>c_G$, $\int_G e^{-C d(g)} \ dg < \infty$.
\end{lem}

\textcolor{black}{Convolution in this article is always left convolution, i.e.~for integrable functions
$\varphi, \psi\in L^1(G)$ we define $\varphi\ast\psi\in L^1(G)$ by
$$ \varphi\ast \psi(g)= \int_G \varphi(x)\psi(x^{-1} g) \ dx\qquad (g\in G)\, .$$
If we denote by $\mathcal{D}'(G)$ the space of distributions, resp.~by $\mathcal{E}'(G)$ the subspace of
compactly supported distributions, then the convolution above naturally extends to distributions provided
one of them is compactly supported, i.e.~lies in $\mathcal{E}'(G)$.}

Denote by $\VG$ the space of left-invariant vector fields on $G$.
It is common to identify the Lie algebra $\g$ with $\VG$ where $X\in \g$ corresponds to the
vector field $\widetilde X$ given by
$$(\widetilde{X} f) (g) = \frac{d}{dt}\Big|_{t=0} f(g \exp(t X))\ \qquad (g\in G, f\in C^\infty(G))\, .$$
We note that the adjoint of $\widetilde X$ on the Hilbert space
$L^2(G)$ is given by
$$\widetilde{X}^*= -\widetilde{X}  - \tr (\ad X)\, ,$$
and  $\widetilde X^*=-\widetilde X$ in case $\g$ is unimodular. Let us
fix an orthonormal basis ${\mathcal B}=\{ X_1, \dots, X_n\} $ of $\g$ with respect
to $\metric$. Then the Laplace--Beltrami operator $\Delta=d^*d$
associated to $\metric$ is given explicitly by
\begin{equation} \label{def Delta}\D = \sum_{j=1}^n (-\widetilde{X_{j}} - \tr (\ad X_j))\ \widetilde{X_{j}}\, .\end{equation}
As $(G, \metric)$ is complete, $\D$ is essentially selfadjoint \textcolor{black}{with spectrum contained in $[0,\infty)$}. We
denote by
$$\SD = \int \lambda \ dP(\lambda)$$
the corresponding spectral resolution. It provides us with a
measurable functional calculus, which allows to define
$$f(\SD)= \int f(\lambda)\ dP(\lambda)$$
as an unbounded operator $f(\SD)$ on $L^2(G)$ with domain
$$D(f(\SD))=\left\{\func \in L^2(G) \mid \int |f(\lambda)|^2\ d\langle P(\lambda)\func, \func\rangle < \infty \right\}.$$
We are going to apply the above calculus to
a certain function space. To do so, for $R'>0$ we define a region
\begin{eqnarray*}
\WRpT = \left\{z \in \C \mid |\operatorname{Im} z| < R' \right\} \cup \left\{z \in
\C \mid |\operatorname{Im} z| < \vartheta |\operatorname{Re} z| \right\}.
\end{eqnarray*}
For $R>0$, $s \in \R$, the relevant function space is then defined as
\begin{eqnarray*}\mathcal{F}_{R,s} &= \{f \in C^\infty(\mathbb{R}, \mathbb{C}) \mid f \text{ even, } \exists \vartheta>0\ \exists R'>R: f \in \mathcal{O}(\WRpT) ,\\ & \qquad  \forall k\in \N: \sup_{z \in \WRpT} |\partial_z^k f(z)|(1+|z|)^{k-s}<\infty\}\ .\end{eqnarray*}
\textcolor{black}{See the Appendix to \S2 in \cite{cgt} for a related space of functions.}\\

The resulting operators $f(\SD)$ are given by a distributional kernel $K_f \in \mathcal{D}'(G\times G)$, $\langle f(\SD)\ \func, \psi\rangle= \langle K_f, \func \otimes \psi\rangle$ for all $\func,\psi \in C^\infty_c(G)$. $K_f$ has the  following properties:
\begin{itemize}
\item smooth outside the diagonal: For $\Delta(G)=\{(g,g) | g \in G\}$, $K_f \in C^\infty(G
\times G \setminus \Delta(G))$,
\item left invariant: $K_f(gx,gy) = K_f(x,y)$,
\item hermitian: $K_f(x,y) = \overline{K_f(y,x)}$\ .
\end{itemize}
By left invariance $f(\SD)$ is a convolution operator with kernel $\kappa_f(x^{-1}y) := K_f(\e,x^{-1}y) = K_f(x,y)$:
\begin{equation} \label{conv0}\langle f(\SD)\ \func, \psi\rangle= \langle K_f, \func \otimes \psi\rangle  = \langle \kappa_f(x^{-1}y) , (\func \otimes \psi)(x,y)\rangle =
\langle \func \ast \kappa_f, \psi \rangle\end{equation}
for all $\func,\psi \in C^\infty_c(G)$.
The distribution $\kappa_f\in \mathcal{D}'(G)$ is a smooth function on $G\setminus \{\e\}$. \textcolor{black}{Because $K_f$ is hermitian, the
kernel $\kappa_f$ is involutive in the sense that
\begin{equation} \label{kappa symmetric}  \kappa_f(x) = \overline{\kappa_f(x^{-1})} \qquad (x \in G).  \end{equation} }
In particular,  $\kappa_f$ is  left differentiable at $x \in G \setminus\{\e\}$, if and only if it is right differentiable at $x$.

\par We define the weighted $L^1$-Schwartz space on $G$ by
$${\mathcal S}_R(G):=\{ f \in C^\infty(G)\mid \forall u, v  \in {\mathcal U} (\g): \ (\tilde u_l \otimes \tilde v_r) f \in L^1(G, e^{Rd(g)} dg) \}\ ,$$
where $\tilde u_l$, resp.~$\tilde v_r$, is the left, resp.~right, invariant differential operator
on $G$ associated with $u, v\in {\mathcal U}(\g)$.

A theorem by Cheeger, Gromov and Taylor
\cite{cgt} allows us to describe the global behavior of $\kappa_f$:
\begin{thm}\label{Kernel}
Let $R,\varepsilon>0$, $s \in \mathbb{R}$ and $f \in \mathcal{F}_{R,s}$. Then $\kappa_f  = \kappa_1 + \kappa_2$, where
\begin{enumerate}
\item $\kappa_1 \in \mathcal{E}'(G)$ is supported in $B_\varepsilon(\e)$, and $K_1(x,y) = \kappa_1(x^{-1}y)$ is the kernel of a pseudodifferential operator on $G$ of order $s$,
\item  $\kappa_2 \in \mathcal{S}_R(G)$.
\end{enumerate}
\end{thm}
\textcolor{black}{Part (1) is the content of Theorem 3.3 in \cite{cgt}. For (2), the pointwise decay of $\kappa_2$ is stated in (3.45)  there, while the Schwartz estimates are obtained as in their Appendix to \S2.}

From part (1) and the kernel estimates for pseudodifferential operators, we  obtain $\kappa_1 \in C^{-s-n-\varepsilon}_c(G)$ for $\varepsilon>0$ \textcolor{black}{small enough, provided $-s-n-\varepsilon>0$. Here $C^{\alpha}_c(G)$ denotes the space of H\"{o}lder continuous functions of order $\alpha>0$, with compact support.}

Applying the theorem to the function $f(z) = (R'^2+z^2)^{-m}$ for $m\in \N$, which lies
in ${\mathcal F}_{R,-2m}$ for any $R<R'$, we conclude the following factorization of the Dirac distribution $\delta_\e$:
\begin{pro}\label{factorize}
Let $R'>R>0$, $m \in \mathbb{N}$. Then
\begin{equation}\label{factorization}\delta_\e =   [(R'^2 + \Delta)^m \delta_\e] \ast \kappa
,\end{equation}
where $\kappa =\kappa_1 + \kappa_2$ \textcolor{black}{has the properties} from Theorem \ref{Kernel} with $s=-2m$.
\end{pro}
\begin{proof}\textcolor{black}{ Set $T:=f(\sqrt{\Delta})$ and $S:=\frac{1}{f}(\sqrt{\Delta})$. Notice that $S(\varphi) = (R'^2 + \Delta)^m\varphi\in C_c^\infty(G)$ and thus
$TS(\varphi)=\varphi$ for all $\varphi \in C_c^\infty(G)$ by the functional calculus. In particular,
\begin{equation} \label{conv1} \varphi= [(R'^2+\Delta)^m\varphi] \ast \kappa\end{equation}
since $T$ is given by right convolution with $\kappa=\kappa_f$, see \eqref{conv0}.
Choose a Dirac sequence $\varphi_n \to\delta_\e$. Passing to the limit in \eqref{conv1} yields
\begin{equation} \label{conv2} \delta_\e =   [(R'^2 + \Delta)^m \delta_\e] \ast \kappa, \end{equation}
as asserted.} \end{proof}

\section{Banach representation of Lie groups}
In this section we briefly recall some basics on Banach representation of Lie groups
and apply Proposition \ref{factorize} to the factorization of  vectors in $E^k$.

\par For a Banach space $E$ we denote by $GL(E)$ the associated
group of \textcolor{black}{topological linear} isomorphisms. By a {\it
Banach representation} $(\pi, E)$ of a Lie group $G$ we understand a group homomorphism
$\pi: G \to GL(E)$ such that the action
$$ G\times E \to E, \ \ (g, v)\mapsto \pi(g)v,$$
is continuous. For a vector $v\in E$ we denote by
$$\gamma_v: G\to E, \ \ g\mapsto \pi(g)v,$$
the corresponding continuous orbit map. Given $k \in \mathbb{N}_0$, the subspace $E^{k} \subset E$ consists of all $v\in E$ for which $\gamma_v \in C^k$. We write $E^\infty = \bigcap_k E^{k}$ and refer to $E^\infty$ as the space of smooth vectors.  Note that all
$E^k$ for $k\in \N_0\cup\{\infty\}$ are $G$-stable.

\begin{rem} \label{B-rep} Let $(\pi, E)$ be a Banach representation.
The uniform boundedness principle implies that
the function
$$w_\pi: G \to \R_+, \ \ g\mapsto \|\pi(g)\|,$$
satisfies the assumptions of Lemma \ref{lem=sm}. \end{rem}
Let
$$c_\pi:=\inf\{c >0\mid \exists C>0: \ w_\pi(g)\leq C e^{c d(g)}\}\, .$$

For $R>0$ we introduce the exponentially weighted spaces
$$\RnG := L^1(G, w_R dg), \ \ w_R(g) = e^{R d(g)}\ .$$
\textcolor{black}{ Notice that $\RnG \subset \mathcal{R}_{R'}(G)$ for $R>R'$ and that the corresponding Fr\'echet algebra
$\mathcal{R}(G):=\bigcap_{R>0} \RnG$ is independent of the particular choice of the metric
$\metric$. }

Denote by $\pi_l$ the left regular representation of $G$ on $\RnG$, and by $\pi_r$ the right regular representation.
A simple computation shows that $\RnG$ becomes a Banach algebra
under left convolution
$$\func*\psi(g)=\int_G \func(x)\ [\pi_l(x)\psi](g) \ dx \qquad (\func, \psi \in \RnG, g\in G)\, $$
for $R>c_G$.

More generally, whenever $(\pi, E)$ is a Banach representation, Lemma \ref{lem=sm} and Remark \ref{B-rep} imply
that
$$\Pi(\func)v:= \int_G \func(g)\ \pi(g)v\ dg \qquad (\func\in \RnG, v\in E)$$
defines an absolutely convergent Banach space valued integral for $R>R_E:=c_\pi+c_G$. Hence the prescription
$$\RnG\times  E\to E, \ \ (\func, v)\mapsto \Pi(\func)v,$$
defines a continuous algebra action  of $\RnG$ (here continuous refers to the continuity of the
bilinear map  $\RnG\times E\to E$).

As an example, the left-right representation $\pi_l\otimes \pi_r$ of $G\times G$ also induces a Banach representation on $\RnG$.

\par Our concern is now with the \textcolor{black}{space of} $k$-times differentiable vectors $\RnG^{k}$
of $(\pi_l\otimes \pi_r, \RnG)$.
 It is clear that $\RnG^{k}$
is a subalgebra of $\RnG$ and that
$$\Pi(\RnG^{k})\ E \subset E^{k}\ ,$$
whenever $(\pi, E)$ is a Banach representation and $R>R_E$.

\begin{thm} Let $R>0$ and $k=2m$ for $m\in \N$. Set $k':=k - \dim G - 1 \geq 0$. Then there exists a
$\kappa\in \RnG^{k'}$ such that: For all Banach representations $(\pi, E)$ with
$R>R_E$ one has the following factorization of $k$-times differentiable vectors
\begin{equation}\label{vfactorization} v =   \Pi (\kappa) d\pi ((R^2 + \Delta)^m)v \qquad (v\in E^k)\, .\end{equation}
\end{thm}

\begin{proof}
Recall the factorization \eqref{factorization} of $\delta_\e$,
$$\delta_\e =    [(R^2 + \Delta)^m \delta_\e]\ast \kappa\ . $$
We claim $\kappa \in \RnG^{k'}$.  \textcolor{black}{Indeed, for $s=-2m$, $n= \dim G$ and $\varepsilon \in (0,1)$, Theorem \ref{Kernel} shows that $\kappa_1 \in C^{2m-\dim G - \varepsilon}_c(G) \subset \RnG^{k'}$ and $\kappa_2 \in \mathcal{S}_R(G) \subset \RnG^{k'}$.
We then obtain that
$$ \gamma_v = [(R^2 + \Delta)^m \gamma_v]\ast \kappa\ ,$$
see also \eqref{conv1}, and evaluation at $g=\e$ gives
$$ v = \gamma_v(\e) = \int_G  \kappa(g^{-1}) \pi (g) d\pi ((R^2 + \Delta)^m)v  \ dg\ . $$
Now recall from \eqref{kappa symmetric} that $\kappa(g)= \overline{\kappa(g^{-1})}$ and that with our choice of $f(z) = (R^2+z^2)^{-m}$ from
before the kernel $\kappa$ is even real. Hence
$$v = \Pi(\kappa)  d\pi((R^2 + \Delta)^m) v\ , $$
as asserted. }
\end{proof}

\begin{cor} \label{cor spectral gap} Let $R>R_E$. Then
$$d\pi\left(R^2+\Delta\right): E^\infty\to E^\infty$$
is invertible.
\end{cor}

\begin{rem} (Spectral gap for Banach representations) We can interpret
Corollary \ref{cor spectral gap} as a spectral gap theorem for Banach representations in terms
of $R_E= c_G + c_\pi$. However, we note that the bound $R>R_E$ can be improved for special classes
of representations. For example,
if $(\pi, E)$ is a unitary representation, then
$$\operatorname{Re} \langle d\pi(\Delta) v, v\rangle \geq 0$$
 for all $v \in E^\infty$, and hence
$d\pi(\Delta) + R^2$ is injective for all $R>0$. Moreover, the Lax-Milgram theorem implies
that $d\pi(\Delta)+R^2$ is in fact invertible. On the other hand, our bound in Corollary
\ref{cor spectral gap} gives information about the convolution kernel of the inverse of $d\pi(\Delta)+R^2$ for $R>c_G$.
\end{rem}

\section{Sobolev norms for Banach representations}

\subsection{Standard and Laplace Sobolev norms}
As before, we let $(\pi, E)$ be a Banach representation. On $E^\infty$, the space
of smooth vectors, one usually defines Sobolev norms as follows.
Let $p$ be the norm underlying $E$. We fix a basis ${\mathcal B}=\{ X_1, \ldots, X_n\}$ of
$\g$ and set
\begin{align*}p_k(v)&:=\Big[\sum_{m_1+\ldots+m_n\leq k} p(d\pi (X_1^{m_1}\cdot \ldots \cdot X_n^{m_n})v)^2 \Big]^{1\over 2}\qquad (v \in E^\infty)\, . \end{align*}
Strictly speaking this notion depends on the choice of the basis ${\mathcal B}$ and
$p_{k, {\mathcal B}}$ would be the more accurate notation. However, a different
choice of basis, say ${\mathcal C}=\{ Y_1, \ldots, Y_n\}$ leads to
an equivalent family of norms $p_{k, {\mathcal C}}$, i.e.~for all $k$ there exist constants
$c_k , C_k>0$ such that
\begin{equation} \label{sandwich} c_k \cdot p_{k, {\mathcal C}}(v) \leq p_{k,{\mathcal B}}(v)\leq C_k \cdot p_{k, {\mathcal C}}(v)\qquad (v\in E^\infty)\, .\end{equation}
Having said this, we now drop  the subscript ${\mathcal B}$ in the definition of $p_k$ and
simply refer to $p_k$ as the {\it standard $k$-th Sobolev norm} of $(\pi, E)$.
Note that $p_k$ is Hermitian (i.e.~obtained from a Hermitian inner product) if $p$ was Hermitian.
\par The completion of $(E^\infty, p_k)$ yields $E^k$. In particular, $(E^k, p_k)$ is a
Banach space for which the natural action $G\times E^k \to E^k$ is continuous, i.e.~defines a Banach representation.
\par The family $(p_k)_{k\in \N}$ induces a Fr\'echet structure on $E^\infty$  (in view of \eqref{sandwich} of course independent of the choice of ${\mathcal B}$) such that the natural
action $G\times E^\infty\to E^\infty$ becomes continuous.

\par Now we introduce a family of {\it Laplace Sobolev norms}, first of even order
$k \in 2 \N_0$, as follows. Let $R>R_E$ and set
\begin{align*}
{}^{\Delta} p_{k}(v)&:= p\left( d\pi\left((R^2+\Delta)^{k/2}\right) v\right) \qquad (v\in E^\infty) \ .
\end{align*}
Of course, a more accurate notation would include $R>0$, i.e.~write ${}^{\Delta,R} p_{k}$ instead of ${}^{\Delta} p_{k}$. In addition, $\Delta$ also depends on the basis ${\mathcal B}$. For purposes of readability we decided to suppress this data in the notation.

\begin{pro} {\rm (Comparison of the families $(p_{2k})_{k\in\N_0}$ and
$({}^{\Delta} p_{2k})_{k \in \N_0}$)} \label{prop Pb}\\
For all $k\in \N_0$ there exist $\textcolor{black}{c_k}$, $C_k>0$ such that for all $v \in E^{\infty}$
$$ \textcolor{black}{c_k \cdot {}^\Delta p_{2k}(v)\leq p_{2k}(v) \leq C_k \cdot {}^\Delta p_{2k+m}(v)\, ,}$$
\textcolor{black}{where $m$ is the smallest even integer greater or equal to $1+\dim G$.}
\end{pro}
\begin{proof}
The first inequality follows directly from the definitions of $p_{2k}$, ${}^{\Delta} p_{2k}$.  The second is a consequence of the factorization \eqref{vfactorization}.
\end{proof}
\begin{rem}\label{regremark} In general it is not true that $p_{2k}$ is smaller than a multiple of
${}^\Delta p_{2k}$. In other words, an index shift as in Proposition \ref{prop Pb},
is actually needed. As an example we consider $E=C_0(\R^2)$ of continuous functions on $\R^2$
which vanish at infinity, endowed with the sup-norm $p(f)=\sup_{x\in \R^2} |f(x)|$.
Then $E$ becomes a Banach representation  for the regular action
of $G=(\R^2, +)$ by translation in the arguments. In this situation there exists a function $u\in
E$ such $\Delta u \in E$ but $\partial_y^2 u\not \in E$, see \cite[Problem 4.9]{gt}. Hence $p_2(u)=\infty$, while
${}^\Delta p_2(u)<\infty$.
\end{rem}

\subsection{Sobolev norms of continuous order $s\in\R$}
\subsubsection{Induced Sobolev norms}
In \cite{bk} Sobolev norms for a Banach representation $(\pi, E)$ were defined for all
parameters $s \in \R$. We briefly recall their construction.

We endow the continuous dual $E'$ of $E$ with the dual norm
$$p'(\lambda):=\sup_{p(v)\leq 1} |\lambda(v)| \qquad (\lambda\in E')\, .$$
For $\lambda \in E'$ and $v \in E^\infty$ we define the matrix coefficient
$$m_{\lambda, v}(g) =\lambda(\pi(g) v) \qquad (g\in G)\ ,$$
which is a smooth function on $G$.  Given  an open relatively compact neighborhood $B\subset G$ of $\e$, diffeomorphic to the
open unit ball in $\R^n$, we  fix $\varphi \in C_c^\infty(G)$ such that $\operatorname{supp} (\varphi)\subset B$
and $\phi(\e)=1$. The function $\phi \cdot m_{\lambda, \phi}$ is
then supported in $B$ and upon identifying $B$ with the open unit ball in $\R^n$, say
$B_{\R^n}$, we denote by $\|\phi \cdot m_{\lambda, v}\|_{H^s(\R^n)}$ the corresponding Sobolev
norm.
We then set
$$Sp_s(v):=\sup_{\lambda \in E'\atop p'(\lambda)\leq 1} \| \phi \cdot m_{\lambda, v}\|_{H^s(\R^n)}
\qquad (v\in E^\infty)\, .$$
In the terminology of \cite{bk} this defines a $G$-continuous norm on $E^\infty$.

\subsubsection{Laplace Sobolev norms}

For $R>R_E$ and $s\in \R$, on the other hand the functional calculus for $\SD$ also gives rise to a $G$-continuous norm on $E^\infty$: We define
\begin{equation} \label{Rost2}{}^{\Delta} p_{s}(v) := p((R^2+\Delta)^{s/2} \gamma_v(g) |_{g=\e})\qquad (v\in E^\infty) \, .\end{equation}

\subsubsection{Comparison results}

\begin{pro}  {\rm (Comparison of the families $(Sp_{s})_{s \geq 0}$ and
$({}^{\Delta} p_{s})_{s \geq 0}$)} \label{prop compare}
Let $R>R_E$. Then for all $s \geq 0$, $\eps>0$, there exist $c_s, C_s>0$ such that for all $v \in E^{\infty}$
$$c_s\cdot Sp_{s}(v) \leq {}^{\Delta} p_{s}(v) \leq C_s\cdot  Sp_{s+\frac{n}{2}+\varepsilon}(v) \ .$$
\end{pro}
\begin{proof}
The first inequality was shown in \cite{bk} for $k \in 2 \mathbb{N}$. It follows for all $s \geq 0$ by interpolation.

For the second  inequality, we apply the standard Sobolev embedding theorem for
$\R^n$ and obtain that
$$\|\phi \cdot m_{\lambda, v} \|_{H^{s+\frac{n}{2}+\varepsilon}(\R^n)} \gtrsim \| \phi\cdot m_{\lambda, v}\|_{C^{s}(B_{\R^n})} \gtrsim |\lambda((R^2+\Delta)^{s/2} \pi(g)v)|_{g=\e}|\ .$$
The assertion follows by taking the supremum over $\lambda \in E'$ with $p'(\lambda)\leq 1$.
\end{proof}

\subsection{Sobolev norms of order $s\leq 0$}
The natural way to define negative Sobolev norms is by duality. For a Banach representation
$(\pi, E)$ with defining norm $p$ and $k\in\N_0$ we let $p_k'$ be the norm
of $(E')^k$ and define $p_{-k}$ as the dual norm of $p_{k}'$, i.e.
$$p_{-k}:= (p_k')' \, .$$
The norm $p_{-k}$ is naturally defined on $((E')^k)'$.  Now observe that the natural inclusion
$(E')^k \hookrightarrow  E'$ is continuous with dense image and thus
yields a continuous dual inclusion $E''\hookrightarrow ((E')^k)'$. The double-dual
$E''$ contains $E$ in an isometric fashion. Hence $p_{-k}$ gives rise to a
natural norm on $E$, henceforth denoted by the same symbol,  and the completion of $E$ with respect to $p_{-k}$ will be denoted by $E^{-k}$.

\begin{rem} In case $E$ is reflexive, i.e.~$E''\simeq E$, the space $E^{-k}$ is isomorphic to the
strong dual of $(E')^k$.
\end{rem}

\par On the other hand we have seen that the families $(p_k)_k$ and $\left({}^\Delta p_k\right)_k$ are equivalent. In this regard we note that
${}^\Delta p_{-k}$ as defined in \eqref{Rost2} coincides with the dual norm of ${}^\Delta p'_k$.

As a corollary of Proposition \ref{prop Pb} (and interpolation to also non-even indices
 $k\in \N_0$) we have:

 \begin{cor} For all $k\in \N_0$ there exist constants $c_k, C_k>0$ such that
 \begin{equation} c_k\cdot p_{-k}(v) \leq {}^\Delta p_{-k}(v) \leq C_k \cdot p_{-k+n+1}(v) \qquad (v\in E^\infty)\, .\end{equation}
 \end{cor}

\end{document}